\documentclass[12pt]{article}

\usepackage{
   amsthm,
   amsmath,
   amssymb,
   graphicx,
   hyperref,
  mathtools,
  fullpage
   }
\usepackage[shortlabels]{enumitem}

\usepackage{mathrsfs}
\usepackage[margin=0pt]{subcaption}
\usepackage{stmaryrd}
\usepackage{tikz-cd}
\usepackage{tikz}

\theoremstyle{plain}
\newtheorem{thm}{Theorem}[section]

\newtheorem{propn}[thm]{Proposition}

\theoremstyle{definition}
\newtheorem{defn}[thm]{Definition}
\newtheorem{rmk}[thm]{Remark}
\newtheorem{ex}[thm]{Example}

\newcommand{\BR}{\mathbb{R}}
\newcommand{\BC}{\mathbb{C}}
\newcommand{\BH}{\mathbb{H}}

\newcommand{\Aut}{\operatorname{Aut}}

\numberwithin{equation}{section}

\begin{document}

\title{On Automorphisms of Complex $b^k$-Manifolds}
\author{Tatyana Barron and Michael Francis}

\maketitle

\begin{abstract}
The $b$-calculus of Melrose is a tool for studying structures on a smooth manifold with a first  order degeneracy at a given hypersurface. In this framework, Mendoza defined  complex $b$-manifolds. In the spirit of work of Scott,  we extend Mendoza's definition to the case of higher-order degeneracies, introducing the notion of a complex $b^k$-manifold for $k$ a positive integer. We then investigate the local and global automorphisms of  complex $b^k$-manifolds. We also propose $b^k$-analogues for some classical spaces of holomorphic functions.
\end{abstract}



\section{Introduction}

Throughout, $M$ is a smooth manifold and $Z \subseteq M$ is a closed hypersurface. We write $C^\infty(M)$ for the ring of smooth, $\BR$-valued functions on $M$  and $\mathfrak{X}(M)$ for the $C^\infty(M)$-module of (smooth) vector fields on $M$. 

Melrose \cite{Melrose} introduced $b$-calculus as an organizational framework for  the study of  differential operators on $M$ with a first order degeneracy along  $Z$. His formalism has significant  applications to index theory on open manifolds.

It is interesting to study ``$b$~versions'' of classical geometries in which the defining data of a given  geometry suffers a degeneracy along $Z$. Various authors have considered $b$-symplectic manifolds, also called log-symplectic manifolds (\cite{GMP}, \cite{GMW3}, \cite{GMW1}, \cite{GMW2}, \cite{LLSS}). Mendoza \cite{Mendoza} defined complex $b$-manifolds and studied the cohomology of the $b$-Dolbeault complex.

Scott \cite{Scott} generalized $b$-calculus by introducing $b^k$-manifolds, where $k$ encodes the order of degeneracy along the hypersurface $Z \subseteq M$. From the point of view of Scott's theory, ordinary $b$-geometry is the case $k=1$. 

In this article, we extend Mendoza's definition in the spirit of Scott's work by  defining what is a complex $b^k$-manifold for $k >1$. We then restrict attention to the (real) two-dimensional case and investigate the local and global automorphisms of complex $b^k$-manifolds. We briefly discuss candidates for function spaces one can attach to a complex $b^k$-manifold. 

\section{Almost-injective Lie algebroids}\label{almostinjective}

Let $M$ be a smooth manifold. The categories of smooth vector bundles on $M$ and finitely-generated, projective $C^\infty(M)$-modules are equivalent by Serre-Swan duality \cite{Swan}. Under this equivalence,  injective maps  of $C^\infty(M)$-modules correspond exactly  to vector bundle maps  that are injective over a dense set. One may therefore regard the following data as equivalent:
\begin{itemize}
\item a projective $C^\infty(M)$-module $\mathcal{F}$ of vector fields on $M$,
\item a vector bundle $A$ over $M$ together with a bundle map $\rho : A \to TM$ that is injective over a dense (open) subset of $M$.
\end{itemize}
One may also deal with the corresponding sheaf of local vector fields:
\begin{align}\label{sheaf}
\mathcal{F}_U  &\coloneq C^\infty(U)\{X|_U : X \in \mathcal{F}\} = \rho(C^\infty(U;A)), \text{ for $U \subseteq M$ open.}
\end{align}

If $\mathcal{F}$ is closed under Lie bracket, then $A$ is naturally a Lie algebroid with bracket coming from the identification $C^\infty(M;A)\cong \mathcal{F}$. Lie algebroids  whose anchor map is injective over a dense set are called \textbf{almost-injective} and appear  in the study of  Stefan-Sussmann singular foliations. See, for example, \cite{Debord} and 5.2.5 in \cite{Crainic-Fernandes}. In this language, the following data are equivalent:

\begin{itemize}
\item an involutive, projective $C^\infty(M)$-module $\mathcal{F}$ of vector fields on $M$,
\item an almost-injective Lie algebroid $A=(A,\rho,[\cdot,\cdot])$ over $M$.
\end{itemize}

\section{The \texorpdfstring{$b$}{b}-tangent bundle}

Two superficially different approaches to $b$-calculus are in use. The original approach of \cite{Melrose} uses a manifold with boundary. It is also common to work  on a manifold without boundary that is equipped with a preferred hypersurface \cite{BLS}. We shall adopt the second approach here.

\begin{defn}
A \textbf{$\mathbf{b}$-manifold} is a smooth 
manifold $M$  together with a  closed hypersurface $Z \subseteq M$ that we refer to as  the \textbf{singular locus}. A \textbf{$\mathbf{b}$-vector field}  is a   vector field on $M$   tangent to $Z$. The $C^\infty(M)$-module of $b$-vector fields is denoted  ${^b\mathfrak{X}}(M)$.  
\end{defn}

\begin{ex}
If $M=\BR^2$ and  $Z=\{0\}\times \BR$, then ${^b\mathfrak{X}(\BR^2)}=\langle x\partial_x,\partial_y\rangle$, the $C^\infty(\BR^2)$-module generated by $x\partial_x$ and $\partial_y$. 
\end{ex}

Above, ${^b\mathfrak{X}(\BR^2)}$ is a free $C^\infty(\BR^2)$-module. In general, for any $b$-manifold $M$,  ${^b\mathfrak{X}}(M)$ is a projective $C^\infty(M)$-module closed under Lie bracket. Thus, referring to Section~\ref{almostinjective}, the following definition makes sense.

\begin{defn}
The  \textbf{$\mathbf{b}$-tangent bundle} of a $b$-manifold $M$ is the almost injective Lie algebroid 
${^bTM}$ associated to the involutive, projective module $^b\mathfrak{X}(M)$.
\end{defn}

\section{\texorpdfstring{$b^k$}{bk}-Tangent bundles}

Fix an integer $k \geq 2$. By analogy with $b$-vector fields ($k=1$),  a $b^k$-vector field  on a manifold $M$ with given closed hypersurface $Z$ ought to be a vector field that is ``tangent to $Z$ to $k$th order''.  It is clear how to best proceed in the model Euclidean example  where $M=\BR^n$, $Z=\{0\}\times \BR^{n-1}$.

\begin{defn}\label{stanbk}
Fix a positive integer $k$. The \textbf{standard module of $\mathbf{b^k}$-vector fields} on $\BR^n$ is $\langle x_1^k \partial_{x_1}, \partial_{x_2},\ldots,\partial_{x_n}\rangle$. That is, the $C^\infty(\BR^n)$-module of vector fields on $\BR^n$ whose first component vanishes to $k$th order on $\{0\}\times\BR^{n-1}$. 
\end{defn}

Unfortunately, the naive notion of ``$k$th order tangency''   used above is not compatible with  smooth coordinate changes.

\begin{ex}
The diffeomorphism $\theta : \BR^2 \to \BR^2$,  $\theta(x,y)=(e^yx,y)$ satisfies $\theta_*(\partial_y) =  x \partial_x +\partial_y$. Thus, although $\theta$ preserves the $y$-axis, pushforward  by $\theta$ does not preserve the standard module of $b^2$-vector fields.
\end{ex}

The above example shows it does not make sense to talk about \emph{the} module of $b^k$-vector fields. Instead, we prescribe as external data a projective module of vector fields on $M$ that is tangent to $Z$ and locally $C^\infty$-conjugate to the model example of Definition~\ref{stanbk}.

\begin{defn}Fix a positive integer $k$. A \textbf{module of $\mathbf{b^k}$-vector fields} for  an $n$-dimensional, smooth manifold $M$  with  a given closed hypersurface $Z \subseteq M$ is  a projective $C^\infty(M)$-module $\mathcal{F} \subseteq \mathfrak{X}(M)$ such that  every $p \in M$ admits an open neighbourhood $U$ and a diffeomorphism $\theta : U \to \BR^n$ satisfying:
\[ \theta_*(\mathcal{F}_U) = 
\begin{cases}
\langle x_1^k \partial_{x_1},\partial_{x_2},\ldots,\partial_{x_n}\rangle & \text{ if } p \in Z\phantom{.} \\
\mathfrak{X}(\BR^n)  & \text{ if } p \notin Z \\
\end{cases} \]
(see  \eqref{sheaf} for notation). 
Such an $\mathcal{F}$ is  involutive; its corresponding    almost-injective Lie algebroid is called the \textbf{$\mathbf{b^k}$-tangent bundle} and denoted ${^{b^k}TM}$. 
\end{defn}

\begin{rmk}
It is tempting to call a pair  $(M,\mathcal{F})$  consisting of a manifold with a given module of $b^k$-vector fields a  \emph{$b^k$-manifold}. However, we shall not do this because  Scott \cite{Scott} has already defined a  $b^k$-manifold  to be a triple $(M,Z,j_Z)$, where $j_Z$ is the $(k-1)$-jet along $Z$ of a defining function for $Z$ (plus an  orientation on $M$ and $Z$). Such a jet induces  a module of $b^k$-vector fields in our sense, namely (Definition~2.5 in \cite{Scott}) the collection $\mathcal{F}$ of vector fields $X$ such that $Xf$ vanishes to $k$th order on $Z$, for any defining function $f$ representing $j_Z$. On the other hand, the assignment $j_Z \mapsto \mathcal{F}$ is neither injective nor surjective. Regarding injectivity, taking $k=2$, $M=\BR^2$, $Z=\{0\}\times\BR$, both $j_Z=x$ and $j_Z=2x$ induce $\mathcal{F}=\langle x^2\partial_x,\partial_y\rangle$. Regarding surjectivity, the module of $b^2$-vector fields $\langle x^2\partial_x , \partial_y + x \partial_x\rangle$ for $M=\BR \times S^1$, $Z=\{0\}\times S^1$, for which local trivializations can be obtained from $\theta(x,y)=(e^yx,y)$, cannot arise from the jet of a defining function. This can be seen from Chapter~5 of \cite{Francis} where  a holonomy invariant is attached to a module of $b^k$-vector fields (there called a \emph{transversely order $k$ singular foliation}).
\end{rmk}

\section{\texorpdfstring{$b$}{b}-Complex and \texorpdfstring{$b^k$}{bk}-complex manifolds}

Recall that a complex structure on  $M$ is given by a  subbundle ${T^{0,1}M}$ of the complexified tangent bundle satisfying: (i) $\BC TM =\overline{{T^{0,1}M}} \oplus {T^{0,1}M}$, and (ii) $T^{0,1}M$ is involutive. This  bundle-theoretic definition is equivalent to the definition via holomorphic charts by the  Newlander-Nirenberg theorem: every point in $M$ admits local coordinates $(x_1,y_1,\ldots,x_n,y_n)$ in which the complex derivatives $\partial_{\overline z}=\frac{1}{2} (\partial_{x_j}+i\partial_{y_j})$, $j=1,\ldots,n$ form a local frame for $T^{0,1}M$. To define complex $b^k$-manifolds, we simply replace $TM$ by ${^{b^k}TM}$. The case $k=1$ was given in \cite{Mendoza}.

\begin{defn}
Fix a positive integer $k$. A \textbf{complex $\mathbf{b^k}$-structure} for a smooth manifold $M$  with a given closed hypersurface $Z\subseteq M$  and module of $b^k$-vector fields $\mathcal{F}$  is a (complex) subbundle ${^{b^k}T^{0,1}M}$ of the complexified $b^k$-tangent bundle satisfying: (i)  $\BC ({^{b^k}TM}) = \overline{^{b^k}T^{0,1}M} \oplus {^{b^k}T^{0,1}M}$, and (ii) ${^{b^k}T^{0,1}M}$ is involutive. It is  equivalent to work with the corresponding bracket-invariant $C^\infty(M,\BC)$-module of complex vector fields   $\mathcal{F}^{0,1}  \subseteq \BC\mathcal{F}$ and, actually, it is the pair  $(M,\mathcal{F}^{0,1})$ that we call a \textbf{complex $\mathbf{b^k}$-manifold}.
\end{defn}

\begin{rmk}\label{recover}
The module $\mathcal{F}^{0,1}$ recovers all the other data at hand. For example, taking real parts recovers  $\mathcal{F}$, and hence ${^{b^k}TM}$. The singular set for the anchor map is the hypersurface $Z$. We define complex $b^k$-manifolds as  pairs  $(M,\mathcal{F}^{0,1})$ in order to make the definition below simple.
\end{rmk}

\begin{defn}\label{isomdef}
An \textbf{isomorphism} $\theta: (M,\mathcal{E}^{0,1}) \to (N,\mathcal{F}^{0,1})$ of two complex $b^k$-manifolds  is a diffeomorphism $\theta : M \to N$ satisfying
$\theta_*(\mathcal{E}^{0,1})=\mathcal{F}^{0,1}$. 
\end{defn}

\begin{rmk}\label{restriso}The $b^k$-tangent bundle of a complex $b^k$-manifold $M$ is isomorphic to the usual tangent bundle away from the singular locus $Z$, so  $M\setminus Z$ is an ordinary complex manifold. Moreover, complex $b^k$-manifold isomorphisms restrict to complex manifold isomorphisms away from the singular loci.
\end{rmk}

\begin{rmk}
Nothing prevents  one from speaking about involutive almost complex structures on arbitrary Lie algebroids. This is done in \cite{Ida-Popescu} and \cite{Popescu}, for example. On the other hand, at this level of generality, it is not clear what the formal integrability condition of involutivity might imply about the existence of a special local normal forms in the spirit of Newlander-Nirenberg. The case of complex $b$-manifolds is discussed in the next section.
\end{rmk}

\section{Remark on the \texorpdfstring{$b$}{b}-Newlander-Nirenberg theorem}

Ideally, complex $b^k$-manifolds should have a single local model depending only on $k$  and the dimension, thus allowing for an equivalent definition in terms of appropriate ``$b^k$-holomorphic charts''. This holds for ordinary complex manifolds (if one likes, the case $k=0$) by the Newlander-Nirenberg theorem. The $k=1$ case  was partially resolved  in Section~5 of \cite{Mendoza};  complex $b$-manifolds do not have ``formal local invariants'' at the boundary. Recently (results to appear elsewhere), the authors have established the full ``no local invariants'' result in the $k=1$ case and expect the same  to hold  for $k \geq 2$. For this reason, we focus on a single model case in our study of automorphisms below. 

\begin{defn}
The  \textbf{standard $\mathbf{b^k}$-complex plane} is the complex $b$-manifold ${^{b^k}\BR^2} \coloneq (\BR^2,\langle L_k\rangle )$, where $\langle L_k\rangle$ is the $C^\infty(\BR^2,\BC)$-module singly-generated by
\begin{align} L_k \coloneq x^k\partial_x + i \partial_y. 
\end{align}
\end{defn}

As stated above, our justification for focusing on these model cases is the following result whose proof will appear elsewhere.

\begin{thm}
Let $M$ be a two-dimensional complex $b$-manifold with singular locus $Z$.  Then, for every $p \in Z$, there is a neighbourhood $U\subseteq M$ of $p$ and an open set $V \subseteq {^{b^k}\BR^2}$ such that $U \cong V$ as complex $b$-manifolds.
\end{thm}

\section{Automorphisms of complex \texorpdfstring{$b^k$}{bk}-manifolds}

We describe the global automorphisms of the standard $b^k$-complex planes by relating complex $b^k$-manifold automorphisms to standard complex automorphisms (see Remark~\ref{restriso}). We first note the following simple fact.

\begin{propn}
Fix a positive integer $k$. Then, 
\[ \Aut({^{b^k}\BR^2})=\{ \theta \in \operatorname{Diff}(\BR^2) : \theta_*( L_k) \sim  L_k \}. \]
Here, $ L_k =x^k \partial_x + i \partial_y$ and $\sim$ denotes equality modulo multiplication by a nowhere-vanishing smooth function. Every automorphism  $\theta$ of ${^{b^k}\BR^2}$ satisfies $\theta_*(Z)=Z$ where $Z=\{0\}\times\BR$. 
\end{propn}
\begin{proof}
From Definition~\ref{isomdef}, an isomorphism ${^{b^k}\BR^2}\to{^{b^k}\BR^2}$ is a diffeomorphism preserving the $C^\infty(\BR^2,\BC)$-module singly-generated by $L_k$, hence the first statement. For the second statement, see Remark~\ref{recover}.
\end{proof}

\begin{ex}\label{flip}
For every positive integer $k$, $(x,y)\mapsto(-x,(-1)^{k+1}y)$ is an order-two automorphism of ${^{b^k}\BR^2}$ that maps the (open) left half-plane $\BR^2_-$ onto the right half-plane $\BR^2_+$. 
\end{ex}

\begin{defn}
We write $\Aut_+({^{b^k}\BR^2})$ for the normal subgroup of $\Aut({^{b^k}\BR^2})$ consisting of automorphisms which preserve the right half-plane $\BR^2_+$. 
\end{defn}

As the full automorphism group $\Aut({^{b^k}\BR^2})$ is the semidirect product product of $\Aut_+({^{b^k}\BR^2})$ by the order-two automorphism of Example~\ref{flip}, we are content to obtain a description of $\Aut_+({^{b^k}\BR^2})$.

\begin{thm}
The group $\Aut_+({^b\BR^2})$  is isomorphic to $\BR\times\BR$. More precisely, it is generated by the following 1-parameter subgroups:
\begin{align*} 
\text{vertical translations:} &&  (x,y) &\mapsto(x,y+t) && t  \in \BR \\
\text{horizontal scalings:} && (x,y) &\mapsto(e^tx,y) &&t \in \BR.
\end{align*}
\end{thm}
\begin{proof}
Let $\BC$ have its usual complex structure. The diffeomorphism
\begin{align*}
\theta : \BR^2_+ \to \BR^2 && \theta(x,y)=\log x + iy
\end{align*}
satisfies $\theta_*(x \partial_x + i \partial_y) = \partial_x + i\partial_y$; it is an isomorphism from $(\BR^2_+,\langle L_1\rangle)$ to $\BC=(\BR^2, \langle L_0\rangle)$.  From Remark~\ref{restriso}, the restriction of any $\alpha \in \Aut_+({^b\BR^2})$ to $\BR^2_+$ pushes forward through $\theta$ to an automorphism of the ordinary complex structure of $\BR^2 = \BC$. The automorphisms of $\BC$ are  the group $\BC^* \ltimes \BC$ of affine transformations. Pulled back through $\theta$, the translation automorphisms of $\BC$ become the automorphisms $(x,y) \mapsto (x+s,e^t y)$ of $(\BR^2_+, \langle  L_1\rangle)$, for $s,t \in \BR$. These extend (by the same formula) to automorphisms of ${^b\BR^2}$. Moreover, these extensions  can be seen to be unique by considering the analogous identification of the left half-plane $(\BR^2_-, \langle  L_1 \rangle )$ with  $\BC$.  

On the other hand, the scaling symmetries of $\BC$ pull back through $\theta$ to the automorphisms $(x,y) \mapsto (e^{s\log x-t y},t\log x + s y)$ of $(\BR^2_+, \langle  L_1\rangle)$, for $s+it \in \BC^*$. Such maps do not extend continuously to $\BR^2$ unless $t =0$ and $s>0$. In the latter case, one obtains the automorphisms $(x,y) \mapsto (sx, y^s)$, $s>0$ of $(\BR^2_+, \langle  L_1\rangle)$ which can be extended continuously to $\BR^2$, but not to diffeomorphisms (excepting  the trivial case $s=1$). 
\end{proof}

\begin{thm}
If $k \geq 2$ is an integer,  the group  $\Aut_+({^{b^k}\BR^2})$ is isomorphic to the ``$ax+b$ group'' $\BR_+ \ltimes \BR$. More precisely, it  is generated by the following 1-parameter subgroups:
\begin{align*} 
\text{vertical translations:} &&  (x,y) &\mapsto(x,y+t) && t  \in \BR \\
\text{hyperbolic transformations:} && (x,y) &\mapsto (e^{-\frac{t}{k-1}}x,e^ty) &&t \in \BR.
\end{align*}
\end{thm}

\begin{proof}
Let $\BH \subseteq \BC$ denote the upper half complex plane $\operatorname{Im}(z) > 0$ with its standard complex structure. The diffeomorphism
\begin{align*}
\theta : \BR^2_+ \to \BH && \theta(x,y) = y +  \tfrac{1}{(k-1)x^{k-1}}i
\end{align*}
satisfies $\theta_*(x^k\partial_x+i\partial_y)= - \partial_y + i\partial_x  \sim \partial_x + i \partial_y$; it is an isomorphism from $(\BR^2_+,\langle  L_k\rangle )$ to $\BH=(\BH,\langle L_0\rangle)$. Referring again to Remark~\ref{restriso}, every element of $\Aut_+(\BR^2,\langle  L_k\rangle)$ determines, by conjugation through $\theta$, an automorphism of $\BH$. The usual group of complex automorphisms of $\BH$ is  $\operatorname{PSL}(2,\BR)$, acting as fractional linear transformations. Using the KAN decomposition for $\operatorname{SL}(2,\BR)$,  $\Aut(\BH)$ is the product of the following three one-parameter subgroups:
\begin{align*}
K &= \{ z \mapsto \tfrac{z \cos t - \sin t}{z\sin t + \cos t} : t \in \BR \} && \text{(elliptic)} \\ 
A &= \{ z \mapsto e^t z : t \in \BR \} && \text{(hyperbolic)} \\  
N &= \{ z \mapsto z+ t : t \in \BR \}   && \text{(parabolic)}.
\end{align*}
Pulled back through $\theta$, the latter two groups of  automorphisms of $\BH$ become the following groups of automorphisms $(\BR^2_+,\langle  L_k \rangle)$:
\begin{align*} 
A' = \{ (x,y) \mapsto  (e^{-\frac{t}{k-1}} x , e^t y) : t \in \BR \},  &&
N' = \{ (x,y) \mapsto (x,y+t) : t \in \BR \}  .
\end{align*}
These groups of automorphisms extend (by the same formulas) to automorphisms of ${^{b^k}\BR^2}$. Their extensions are also unique (by considering their action on the left half-plane as well). On the other hand, the automorphisms  of $(\BR^2_+,\langle L_k\rangle)$ arising as the pullback through $\theta$ of $z \mapsto \tfrac{z \cos t - \sin t}{z\sin t + \cos t}$ for 
$t$ not an integer multiple of $\pi$ do not extend continuously to $\BR^2$. Indeed, an elementary computation shows that, if the imaginary part of $z$ goes to $+\infty$ and the real part of $z$ is held fixed, then $\tfrac{z \cos t - \sin t}{z\sin t + \cos t} \to \cot(t)$. It follows, if $\alpha_t$ is the corresponding automorphism of $(\BR^2,\langle L_k\rangle)$, that $\alpha_t(x,y)$ approaches $(+\infty, \cot(t))$ as $x \to 0^+$ and $y$ is held fixed. In particular, $\alpha_t$ does not admit a continuous extension to $\BR^2$.
\end{proof}

\begin{rmk}
We note that the action of $\Aut_+({^{b^k}\BR^2})$ on the singular locus  $Z=\{0\} \times \BR$ is faithful for $k\geq 2$, but not for $k=1$. 
\end{rmk}

We  conclude this section with an  example of a family of local automorphisms of ${^b\BR^2}$ that do not extend to a global ones. 

\begin{ex}
Let $\Omega \subseteq {^b\BR^2}$ be the strip $-\frac{\pi}{2}<y<\frac{\pi}{2}$. The function $u : \BR^2 \to \BC$ given by $u(x,y)=xe^{iy}$ 
maps $\Omega \setminus ( \{0\} \times \BR)$ diffeomorphically onto $\BC \setminus \BR$ and satisfies $u_*(x \partial_x + i\partial_y) = \overline z \partial_{\overline z}$. Thus, 
$u_*( xe^{iy}(x \partial_x + i\partial_y)) = \overline z^2 \partial_{\overline z}$ and, after taking real parts, 
\[ u_*(x^2 \cos x \partial_x+x \sin y \partial_y) =  (x^2-y^2)\partial_x + 2xy \partial_y. \]
The vector feld on the right is the generator of the 1-parameter group of Mobius transformations  $z \mapsto \frac{z}{1-tz}$, $t \in 
\BR$ which is  complete when considered as a flow on $\BC \setminus \BR$ (or, of course, the Riemann sphere). Pulling back through $u$,  one checks that the flow of $x^2 \cos x \partial_x+x \sin y \partial_y$ defines a 1-parameter group of complex $b$-manifold automorphisms of $(\Omega, \langle L_1 \rangle)$.
\end{ex}

\section{Function spaces}

A \textbf{$\mathbf{b}$-holomorphic function} on ${^b\BR^2}$ is a smooth function $f:\BR^2 \to \BC$ (perhaps only locally defined) satisfying $ L_1 f= 0$ where $ L_1 = x\partial_x+i\partial_y$.

\begin{ex}
The function $u:{^b\BR^2} \to \BC$, $u(x,y)=xe^{iy}$ is $b$-holomorphic. More generally, $h \circ u$ is $b$-holomorphic near $(0,0)$ for any holomorphic function $h$ defined near $0$. There exist nonanalytic $b$-holomorphic functions defined near $(0,0)$ as well;  $f(x,y)=\exp(\frac{-1}{u(x,y)})$ for $x>0$,  $f(x,y)=0$ for $x \leq 0$ defines a $b$-holomorphic function on $\{ (x,y)\in{^b\BR^2} : -\frac{\pi}{4}<y<\frac{\pi}{4}\}$.  
\end{ex}

It is interesting to contemplate potential ``$b$ analogues'' for classical spaces of holomorphic functions such as 
Segal-Bargmann space, (weighted) Bergman space, Hardy space, etc.  Let us define ad hoc \textbf{$\mathbf{b}$-Segal-Bargmann space} to be  the space of $b$-holomorphic functions  on ${^b\BR^2}$ whose restriction  to the right half-plane $\BR^2_+$  pushes  forward through the isomorphism $(x,y)\mapsto (\log x, y) : (\BR^2_+,\langle L_1\rangle) \to \BC$ used earlier to an element  of the classical Segal-Bargmann space (entire functions $f$ on $\BC$ satisfying $\int_\BC |f(z)|^2 e^{-|z|^2} \ dA<\infty$). 

\begin{ex}
The function $u(x,y)=xe^{iy}$ belongs to the $b$-Segal-Bargman space (it pushes forward to the exponential function). Considering powers of $u$, we conclude the $b$-Segal-Bargmann space is infinite-dimensional.
\end{ex}

\bibliographystyle{abbrv}
\bibliography{MyBibLib}

\newcommand{\noopsort}[1]{} \newcommand{\printfirst}[2]{#1} \newcommand{\singleletter}[1]{#1} \newcommand{\switchargs}[2]{#2#1}
\begin{thebibliography}{10}

\bibitem{BLS}
M.~Braverman, Y.~Loizides, and Y.~Song.
\newblock Geometric quantization of {$b$}-symplectic manifolds.
\newblock {\em J. Symplectic Geom.}, 19(1):1--36, 2021.

\bibitem{Crainic-Fernandes}
M.~Crainic and R.~L. Fernandes.
\newblock Integrability of {L}ie brackets.
\newblock {\em Ann. of Math. (2)}, 157(2):575--620, 2003.

\bibitem{Debord}
C.~Debord.
\newblock Holonomy groupoids of singular foliations.
\newblock {\em J. Differential Geom.}, 58(3):467--500, 2001.

\bibitem{Francis}
M.~Francis.
\newblock {\em Groupoids and Algebras of Certain Singular Foliations with Finitely Many Leaves}.
\newblock PhD thesis, Pennsylvania State University, 2021.

\bibitem{GMP}
V.~Guillemin, E.~Miranda, and A.~R. Pires.
\newblock Symplectic and {P}oisson geometry on {$b$}-manifolds.
\newblock {\em Adv. Math.}, 264:864--896, 2014.

\bibitem{GMW1}
V.~Guillemin, E.~Miranda, and J.~Weitsman.
\newblock Desingularizing {$b^m$}-symplectic structures.
\newblock {\em Int. Math. Res. Not. IMRN}, (10):2981--2998, 2019.

\bibitem{GMW3}
V.~W. Guillemin, E.~Miranda, and J.~Weitsman.
\newblock Convexity of the moment map image for torus actions on {$b^m$}-symplectic manifolds.
\newblock {\em Philos. Trans. Roy. Soc. A}, 376(2131):20170420, 6, 2018.

\bibitem{GMW2}
V.~W. Guillemin, E.~Miranda, and J.~Weitsman.
\newblock On geometric quantization of {$b^m$}-symplectic manifolds.
\newblock {\em Math. Z.}, 298(1-2):281--288, 2021.

\bibitem{Ida-Popescu}
C.~Ida and P.~Popescu.
\newblock On almost complex {L}ie algebroids.
\newblock {\em Mediterr. J. Math.}, 13(2):803--824, 2016.

\bibitem{LLSS}
Y.~Lin, Y.~Loizides, R.~Sjamaar, and Y.~Song.
\newblock Symplectic reduction and a darboux-moser-weinstein theorem for lie algebroids, 2022.

\bibitem{Melrose}
R.~B. Melrose.
\newblock {\em The {A}tiyah-{P}atodi-{S}inger index theorem}, volume~4 of {\em Research Notes in Mathematics}.
\newblock A K Peters, Ltd., Wellesley, MA, 1993.

\bibitem{Mendoza}
G.~A. Mendoza.
\newblock Complex {$b$}-manifolds.
\newblock In {\em Geometric and spectral analysis}, volume 630 of {\em Contemp. Math.}, pages 139--171. Amer. Math. Soc., Providence, RI, 2014.

\bibitem{Popescu}
P.~Popescu.
\newblock Poisson structures on almost complex {L}ie algebroids.
\newblock {\em Int. J. Geom. Methods Mod. Phys.}, 11(8):1450069, 22, 2014.

\bibitem{Scott}
G.~Scott.
\newblock The geometry of {$b^k$} manifolds.
\newblock {\em J. Symplectic Geom.}, 14(1):71--95, 2016.

\bibitem{Swan}
R.~G. Swan.
\newblock Vector bundles and projective modules.
\newblock {\em Trans. Amer. Math. Soc.}, 105:264--277, 1962.

\end{thebibliography}

\end{document}